\documentclass[a4paper,12pt,reqno]{amsart}
\usepackage{amsmath,amsthm,amsfonts,amssymb}
\usepackage[mathscr]{eucal} % Para letras Q, R, S da equação quadrática - Talvez alterar!!
\usepackage{indentfirst}
\usepackage[authoryear,round]{natbib}
\usepackage[a4paper,top=3cm,bottom=3cm,left=3cm,right=3cm]{geometry}
\usepackage{setspace} % Delete in final version.
\usepackage{caption}
\usepackage{subcaption}
\usepackage{microtype}
\usepackage{xcolor}
\usepackage{xspace}
\usepackage{graphicx}
\usepackage{hyperref}
\usepackage{dsfont}
\usepackage{accents}
\usepackage{color}
\usepackage[normalem]{ulem}

\theoremstyle{plain}
\newtheorem{teo}{Theorem}[section]

\newtheorem{lem}[teo]{Lemma}

\theoremstyle{definition}
\newtheorem{defn}[teo]{Definition}

\theoremstyle{remark}
\newtheorem*{obs}{Remark}

\theoremstyle{remark}
\newtheorem*{obss}{Remarks}

\numberwithin{equation}{section}
\numberwithin{figure}{section}
\numberwithin{table}{section}

\hypersetup{colorlinks,
linkcolor={darkblue},
citecolor={darkblue},
urlcolor={darkblue}}%

\newcommand{\bbZ}{\mathbb{Z}}

\newcommand{\Td}{\mathbb{T}_d}
\newcommand{\Km}{\mathbb{K}_m}
\newcommand{\Dmd}{\mathbb{D}_{m,d}}
\newcommand{\Dd}{\mathbb{D}_{2,d}}
\newcommand{\bbP}{\mathbb{P}}
\newcommand{\bbE}{\mathbb{E}}
\newcommand{\raiz}{\varnothing}
\newcommand{\dist}{\textnormal{dist}}

\newcommand{\cR}{\mathcal{R}}
\newcommand{\cG}{\mathcal{G}}
\newcommand{\cX}{\mathcal{X}}
\newcommand{\FM}[2]{\ensuremath{\textnormal{FM}(#1,#2)}}
\newcommand{\MFM}[2]{\ensuremath{\textnormal{MFM}(#1,#2)}}
\newcommand{\s}[2]{\ensuremath{{#1}\rightarrow{#2}}}
\newcommand{\ns}[2]{\ensuremath{#1 \nrightarrow #2}}
\newcommand{\m}[2]{\ensuremath{#1 \leadsto #2}}
\newcommand{\cam}[2]{\ensuremath{{#1}\stackrel{\textit{\scriptsize{c\,}}}{\rightarrow}{#2}}}
\newcommand{\ncam}[2]{\ensuremath{#1 \stackrel {\textit {\scriptsize {c \,}}}{\nrightarrow} #2}}
\newcommand{\cvis}{\mathcal{T}}
\newcommand{\fA}{A(\alpha, p)}
\newcommand{\fB}{B(\alpha, p)}
\newcommand{\fC}{C(\alpha, p)}
\newcommand{\fphi}{\phi(r)}
\newcommand{\fgamma}{\gamma(\alpha, p)}
\newcommand{\fQ}{\mathscr{Q}(\alpha, p)}
\newcommand{\fR}{\mathscr{R}(\alpha, p)}
\newcommand{\fS}{\mathscr{S}(\alpha, p)}
\newcommand{\fr}{r(p,m,d)}

\newcommand{\fpsi}{\psi(\alpha,p)}

\newcommand{\dado}{{\,}|{\,}}

\newcommand{\ind}{\mathds{1}}

\definecolor{bordo}{rgb}{0.8,0,0.3}
\colorlet{darkblue}{blue!70!black}

\begin{document}

\title[Frog model on free products of complete graphs]{The frog model with death and drift on free products of complete graphs}

\author[N.~L. Garcia]{Nancy Lopes Garcia}
\address[N.~L. Garcia and E. Lebensztayn]{Instituto de Matem\'atica, Estat\'istica e Computa\c{c}\~ao Cient\'ifica, Universidade Estadual de Campinas (UNICAMP), CEP 13083-859, Cam\-pi\-nas, SP, Brasil.}
\email[N.~L. Garcia]{nancyg@unicamp.br}

\author[E. Lebensztayn]{Elcio Lebensztayn}
\email[E. Lebensztayn]{lebensz@unicamp.br}

\author[J. Utria]{Jaime Utria}
\address[J. Utria]{Instituto de Matem\'atica e Estat\'istica, Universidade Federal Fluminense (UFF), CEP 24210-201, Niter\'oi, RJ, Brasil.}
\email[J. Utria]{jutria@id.uff.br}

\thanks{This study was financed by the S\~ao Paulo Research Foundation (FAPESP), Brazil. Processes Number \#2023/13453-5 and \#2025/04134-9.}

\subjclass[2020]{Primary ; Secondary .}
\date{\today}

\begin{abstract} We study the frog model with death and drift on $\Dmd$, the free product of $d+1$ copies of the complete graph of order $m$.
Active and inactive particles are located at the vertices of $\Dmd$. Each active particle performs a $\alpha$-biased random walk towards the root of $\Dmd$, dying after a random lifetime with a geometric distribution of parameter $1-p$.
Each inactive particle remains dormant until an active particle visits its location.
We present conditions on the parameters $\alpha$ and $p$ for the process to die out almost surely and to survive with positive probability. Our proofs are based on comparisons of the model with simple and multi-type branching processes.    
\end{abstract}

\maketitle 

\setstretch{1.4} % Delete in final version.

\section{Introduction}
\label{S: Introduction}

We consider the \emph{frog model with death and drift} on a finitely generated group, extending the model beyond classical lattices and regular trees. As a first step, we focus on $\Dmd$ for $m \geq 2$ and $d \geq 2$, namely, the \emph{free product} of $d+1$ copies of the complete graph of order $m$, viewed as the Cayley graph of a finitely generated group. A vertex of $\Dmd$ is singled out (the identity element of $\Dmd$) and is called its root.
Initially, there is a single active frog at the root and one sleeping frog at each non-root vertex. At each instant of time, each active particle may disappear with probability $1-p$. Active frogs move toward the root with probability $\alpha$; otherwise, they move away from the root to a uniformly sampled neighbor vertex. Frogs at the root that survive stay in place with probability $\alpha$ or move away from the root to a uniformly sampled child vertex with probability $1-\alpha$. When an active particle hits a vertex containing a sleeping one, the latter is activated, and starts its own life and $\alpha$-biased random walk, performing exactly the same dynamics, independent of everything else.  We denote this model by $\FM{\Dmd}{\alpha,p}$. We say that a particular realization of the frog model \emph{survives} if there is at least one active particle at every instant of time. Otherwise, we say that it \emph{dies out}. Here we are interested in answering the following question: for which values of the pair $(\alpha,p)$ does the $\FM{\Dmd}{\alpha,p}$ survive with positive probability or die out almost surely?

For the frog model with death, but not drift, the question of phase transition at which the model survives with positive probability was first addressed by \citet{AMP-PT}, with focus on hypercubic lattices $\bbZ^d$ for $d \geq 2$ and on homogeneous trees $\Td$ of degree $d+1 \geq 3$, whereas on $\mathbb{Z}$, the authors showed that there is no phase transition. \citet{FMS-Mono} show that the critical probability of the frog model is not a monotonic function of the graph. Recently, \citet{CM-2025frog} study the frog model in random environment on $\mathbb{Z}$ in the case that $p$ follows a beta distribution with parameters $a$ and $b$, showing that the model undergoes a phase transition with respect to survival at $b =1/2$.  
To the best of the authors’ knowledge, the only paper dealing with the frog model on a Cayley graph of a finitely generated group (with the exception of $\mathbb{Z}^d$) is \citet{coletti2021asymptotic}; in this paper, they proved a shape theorem result for the model with no death and no drift.
The question of phase transition for a continuous-time variant of the frog model with deterministic lifetimes on general vertex-transitive graphs is investigated in \citet{angel2026frog}.
Our paper is closely related to \citet{LMP-IUB, GR-FMRT, LU-NUB, LU-PTBT, GP}, all of which focus on the derivation of bounds for the critical probability of the frog model with geometric lifetimes on trees. The novelty here is that we incorporate both death and drift parameters and consider the process evolving on a more complex graph structure, allowing for the presence of cycles rather than restricting to trees.
By doing this, we provide a more robust framework that generalizes the bounds previously found in the literature.
% As far as we know, under this setting, the derivation of bounds for the critical probability of the frog model has never been studied in the literature.

For the frog model without death ($p=1$), in which all active frogs live perpetually, a fundamental problem is whether the root is visited by infinitely many particles. The first published paper is due to \citet{TW}, where it was referred to as the ``egg model". They proved that, almost surely infinitely many frogs visit the root of $\bbZ^d$, for $d \geq 3$. \citet{HJJ-RT} establish that there is a phase transition in the dimension of the $d$-ary tree, by proving recurrence for $d=2$ and transience for $d\geq 5$; the cases $d=3$ and $d=4$ remain open. \citet{CNTSF-RTFD} consider the frog model with drift on $\bbZ^d$ and present conditions on the parameters for recurrence and transience. \citet{BNYMS-FMTD} study the model on $d$-ary trees, providing a uniform upper bound on the minimal drift so that the one-per-site frog model is recurrent. For the question of recurrence and transience for the frog model with drift and death on $d$-ary trees, see \citet{ahmed2025frog}.  
 
\section{Formal definitions}% and main results}

We begin by giving some basic definitions from Geometric Group Theory.
\subsection*{Graphs} Let $G = (V, E)$ be an infinite connected locally finite graph, with vertex set $V$ and edge set $E$. We denote an unoriented edge with endpoints $u$ and $v$ by $uv$. Vertices $u$ and $v$ are said \emph{neighbors} if they belong to a common edge $uv$; we denote this by $u \sim v$. The \emph{degree} of a vertex is the number of neighbors. A \emph{path} of length $n$ from $u$ to $v$ is a sequence $u=u_0, \ldots, u_n = v$ of vertices such that $u_i \sim u_{i+1}$. The \emph{graph distance} $\dist(u,v)$ between $u$ and $v$ is the minimal length of a path connecting the two vertices. A path from $u$ to $v$ with length $\dist(u,v)$ is called a \emph{geodesic}. A graph is called \emph{geodetic} if there exists a unique geodesic between every pair of vertices. We say that $G$ is $M$-regular if all vertices have the same degree $M$. A \emph{tree} is a connected graph without cycles or loops, where by a cycle in a graph we mean a sequence of neighboring vertices $u_0, \cdots, u_n$, $n \geq 3$, with no repetitions besides $u_n = u_0$. A \emph{cut-vertex} is a vertex whose removal will disconnect the graph.

\subsection*{Free products of groups and Cayley graphs} Let $\cG$ be a finitely generated infinite group with identity element $\raiz$ and let $S$ be a symmetric set of generators of $\cG$. We usually write the group operation multiplicatively. 
The \emph{Cayley graph} $\cX(\cG,S)$ of $\cG$ with respect to $S$ has vertex set $\cG$, and two elements $u,v \in \cG$ are neighbors if $u^{-1}v\in S$. The resulting graph is connected, locally finite, and homogeneous (all vertices have the same degree $|S|$). We call the vertex $\raiz$ the \emph{root} of the graph.
Let $(\cG_i,\raiz_i)$, $i \in \mathcal{I}$ a finite family of groups with roots $\raiz_i \in \cG_i$. We define the \emph{free product} $\displaystyle (\cG,\raiz) = \mathop{*}_{i \in I} (\cG_i,\raiz_i)$ as follows. We connect all the $\cG_i$ by identifying all the roots $\raiz_i$ in a single common root $\raiz$ while keeping the rest of $\cG_i$ disjoint. The free product consists of all words $u_1\cdots u_n$ of arbitrary finite length $n \geq 0$, where each letter $u_i \in \cG_i' = \cG_{i}\setminus \{\raiz_i\}$, such that no two successive letters come from the same $\cG_i'$. Formally, let $\mathfrak{i}: \cG_i' \to \mathcal{I}$ be defined by $\mathfrak{i}(u) = i$; the free product is the set
%let $\cG_i' = \cG\setminus \{\raiz_i\}$ for every $i \in \mathcal{I}$
\begin{equation}
\label{E: free group}
\cG = \{u_1u_2 \cdots u_n: n\geq 0, u_j \in \bigcup_{i\in\mathcal{I}} \cG_i', \mathfrak{i}(u_j) \neq \mathfrak{i}(u_{j+1})\};
\end{equation}
the empty word $\raiz$ is obtained when $n = 0$. Moreover, $\cG$ is a group with the identity element $\raiz$ and the group operation is concatenation with possible cancellation in the middle, in order to get the reduced form as in \eqref{E: free group}. Furthermore, if $\cX_i = \cX_i(\cG_i,S_i)$ is the Cayley graph of $\cG_i$ with respect to $S_i$, then $\cX = \cX(\cG,S)$ is the Cayley graph of $\cG$ with respect to the generating set $S = \bigcup_{i\in\mathcal{I}} S_i$. 
To construct the Cayley graph of the free product given in \eqref{E: free group}, we proceed as follows: 
\begin{enumerate}
\item At the root $\raiz$, the $\cX_i$ are glued at their respective roots. 
\item At each other vertex of $\cX_i$, we attach copies of all the $\cX_j$, $j\neq i$, by their roots.
\item Inductively, at each new vertices $u_iu_j$ we then attach copies of the $\cX_k$, $k \neq j$.
\end{enumerate}
 
Thus, $\cX$ has a \emph{cactus graph} structure whose leaves are copies of the $\cX_i$. Figure \ref{Fig: free-products of two graphs} shows a piece of the Cayley graphs of two free products of groups. 

\subsection*{Frog model with death and drift on free products of complete graphs}
To define the frog model on a free product of groups, we need a few definitions. We start by describing the state space in which the model lives.
Throughout this paper, for \(m \geq 2\) and \(d \geq 1\), we denote by \(\Km\) the complete graph with $m$ vertices and by \( \Dmd := \Km * \cdots * \Km \) the free product of $(d+1)$ copies of $\Km$. We view $\Dmd$ as the Cayley graph of the group $\bbZ_m \ast \cdots \ast \bbZ_m$ ($d+1$ times), where $\bbZ_m$ is the cyclic group of order $m$ with respect to the generating set $S_m = \bbZ_m \setminus \{\raiz\}$. Note that \(\Dmd \) is a geodetic and $\Delta$-regular graph in which each vertex has degree $\Delta = (d+1)(m-1)$. 
In particular, $\mathbb{D}_{2,1}$ is isomorphic (as a graph) to $\mathbb{Z}$, and $\Dd$ is the homogeneous tree $\Td$ of degree $(d+1)$. Figure \ref{Fig: D4_2} shows a piece of the Cayley graph $\mathbb{D}_{4,2}$; the resulting graph is a \emph{tree of cliques}. With a slight abuse of notation, we denote both the graph and its vertex set by $\Dmd$.
%With typical abuse of notation, we denote by $\Dmd$ the graph itself and its set of vertices. 
Now, we define a partial order on the vertices of $\Dmd$ as follows: for $u,v \in \Dmd$, we write $u \leq v$ if $u$ lies on the unique geodesic from $\raiz$ to $v$; $u < v$ if $u \leq v$ and $u \neq v$, $u$ (resp.~$v$) is called \emph{predecessor} (resp.\ \emph{successor}) of $v$. For any non-root vertex $u$, we distinguish three \emph{types} of neighbors: the \emph{parent} of $u$, the \emph{children} of $u$, and the \emph{siblings} of $u$. More precisely, the parent of $u$, denoted by $u^-$, is the unique predecessor of $u$ along the geodesic to the root; a child of $u$, denoted by $u^+$ is any neighbor of $u$ such that $\mathrm{dist}(\raiz,u^+) = \mathrm{dist}(\raiz,u) + 1$; and a sibling $u^*$ is any remaining neighbor of $u$ (that is, any vertex at the same \emph{level} of $u$, i.e., vertices satisfying $\dist(\raiz,u*) = \dist(\raiz,u))$. Note that any non-root vertex $u$ has one parent, $d(m-1)$ children, and $m-2$ siblings. Moreover, the root is called the \emph{ancestral} vertex since it has only children. 

\begin{figure}[ht]
\includegraphics[scale = 0.7]{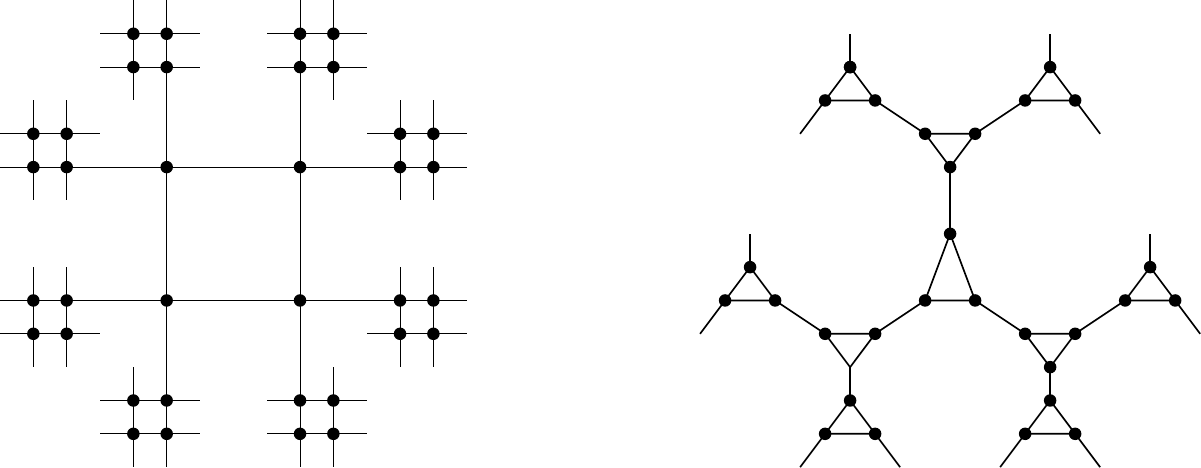}
\caption{Left: structure of $\mathbb{K}_2 \ast \mathbb{K}_2 \ast \mathbb{Z}_4$. Right: structure of  $\mathbb{Z}_2 * \mathbb{Z}_3$.}
\label{Fig: free-products of two graphs}
\end{figure}

In general, for a \emph{simple symmetric random walk} (SSRW) on a finitely generated group $\cG$ with generating set $S$, we mean a SSRW on the Cayley graph $\cX(\cG, S)$, that is, is a Markov Chain with state space $\cG$ and one-step transition probabilities given by $\pi(u,v) = 1/|S|$, $u\sim v$. 
Also, we consider \emph{$\alpha$-biased random walks} on $\cG$, in which particles  move according to the following one-step transition probabilities, which depend on the parameter $\alpha \in [0,1]$:
\begin{equation}\label{F: drift-RW}
\begin{array}{cc}
\begin{aligned}
\pi_{\alpha}(u,v) = \begin{cases}
\alpha, & v=u^-,\\
\frac{1-\alpha}{\Delta -1}, & v \neq u^-.
\end{cases} 
\end{aligned} %\quad\text{and} & \quad
% \begin{aligned}
% \tilde{\pi}_{\alpha}(\raiz,v) = \begin{cases}
% \alpha, & v=\raiz,\\
% \frac{1-\alpha}{\Delta}, & \raiz \sim v.
% \end{cases} 
% \end{aligned}
\end{array}
\end{equation}
%Besides frogs at the root stays there with probability $\alpha$ and choose a neighbor vertex with probability $(1-\alpha)/\Delta$.
Note that if $\alpha = 1/\Delta$, we obtain the simple symmetric random walk on $\cG$. For more details on random walks on groups and graphs, we refer the reader to \citet{woess2000random}.

In the present article, we study the frog model with death and drift on $\Dmd$ for $m \geq 2$ and $d \geq 2$, in which each active frog placed at any non-root vertex (if it survives) chooses a neighbor according to \eqref{F: drift-RW}. To define the model in a formal way we consider $\{\Xi_p^u: u \in \Dmd\}$, $\{(S_n^u)_{n \geq 0}: u \in \Dmd\setminus\{\raiz\}\}$ independent collections of i.i.d.\ random variables. Here $\Xi_p^u$ is a geometric random variable with parameter $1-p$ and stands for the lifetime of the particle located at $u$, $(S_n^u)_{n\geq 0}$ is a discrete time $\alpha$-biased random walks on $\Dmd$ starting from $u\neq \raiz$. Active particles at the root that survive stay in place with probability $\alpha$ or jump away from the root to a uniformly chosen child with probability $1-\alpha$.

We define the (virtual) \emph{range} of the vertex $u$ as the set of vertices visited by the particle originally placed at vertex $u$ during its lifetime, 
\[ \mathcal{R}_u = \{S_n^x : 0 \leq n < \Xi_p^u\}. \]
The range becomes ``real" in the case when $u$ is actually visited (and thus the sleeping frog from there is awakened). Now, for $u, v \in \Dmd$, we introduce the following notations: $\s{u}{v}$ if $v \in \mathcal{R}_u$, and $\ns{u}{v}$, otherwise.
We denote by $\FM{\Dmd}{\alpha,p}$ the frog model on the Cayley graph $\Dmd$ with drift and survival parameters $\alpha$ and $p$, respectively. For the symmetric case, i.e., $\alpha = 1/\Delta$, we denote the model by $\FM{\Dmd}{p}$.

We emphasize that we are, in fact, dealing with a \emph{long-range oriented percolation model} defined as follows: For each pair $(u,v) \in \Dmd \times \Dmd$ with $u \neq v$ we declare the oriented edge from $u$ to $v$ \emph{open} if $\s{u}{v}$ and \emph{closed} otherwise.

% Let us consider the following definition.
% \begin{defn}\label{D: survival/extinction}
%    We say that a particular realization of the frog model \emph{survives} if there is at least one active particle at every instant of time. Otherwise, it \emph{dies out}. 
% \end{defn}
%From Definition \ref{D: survival/extinction} 
We conclude that the survival of $\FM{\Dmd}{\alpha,p}$ is equivalent to the infiniteness of the cluster of $\raiz$ in the percolation model described above. Therefore, the frog model survives if there exists an infinite sequence of vertices $\raiz = u_0,u_1,\ldots$, in such a way that $\s{u_j}{u_{j+1}}$, for any $j\geq 0$. This approach plays a central role in the proofs of most of the results in this paper. 

By a coupling argument, we see that $\bbP(\FM{\Dmd}{\alpha,p} \text{ survives})$ is a non-decreasing function of $p$, and therefore we define the \emph{critical probability} as
\[
p_c(\Dmd,\alpha) = \inf\{p: \bbP(\FM{\Dmd}{\alpha, p} \text{ survives}) > 0\}.
\]

As usual, we say that the model exhibits a phase transition if $p_c(\Dmd, \alpha) \in (0,1)$. In the symmetric case, we omit the dependence on $\alpha$ and write $p_c(\Dmd)$ for the critical probability.

% \begin{figure}[ht]
% \includegraphics[scale=0.8]{D4_2.pdf}
% \caption{Structure of $\mathbb{D}_{4,2} = \mathbb{K}_4 \ast \mathbb{K}_4 \ast \mathbb{K}_4$ (the root is highlighted in red).}
% \label{Fig: D4_2}
% \end{figure}

\begin{figure}
    \centering
    \includegraphics[width=0.5\linewidth]{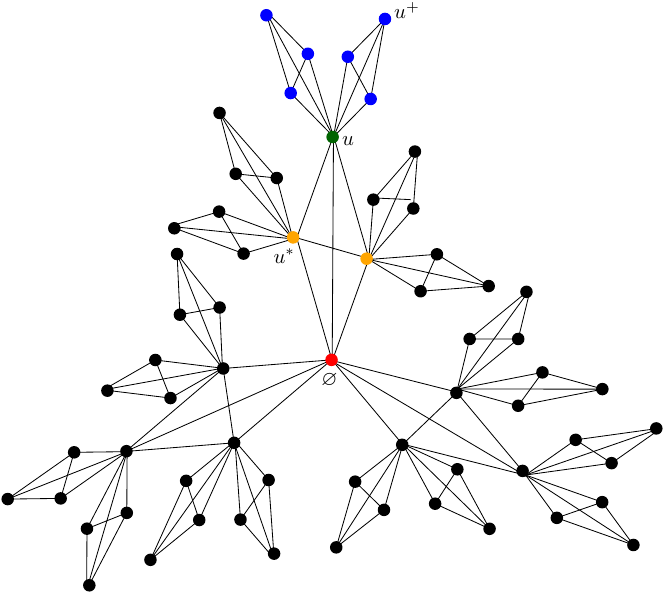}
    \caption{Structure of $\mathbb{D}_{4,2} = \mathbb{K}_4 \ast \mathbb{K}_4 \ast \mathbb{K}_4$. The root (parent of vertex $u$) is shown in red, the vertex $u$ in green, its children in blue, and its siblings in orange.}
    \label{Fig: D4_2}
\end{figure}

\section{Main results}

First, we state the results for the general case $(\alpha,p) \in [0,1]$.
\begin{teo}
\label{T: Extinction}
Let $m \geq 2$ and $d \geq 2$ be fixed integers. 
We define the function
\begin{equation}
\label{F: LE}
\fpsi = \frac{p}{2} \left(2-\alpha +\sqrt{(2-\alpha)^2-\frac{4 (1-\alpha)((m-2)(2-\alpha)+\alpha d(m-1))}{(d(m-1)+m-2)^2}}\right).
\end{equation}
% and consider the set
% \[ \cExt = \{(\alpha, p): \fpsi < 1\}. \]
Then, \(\FM{\Dmd}{\alpha,p}\) dies out almost surely for $(\alpha, p)$ such that $\fpsi < 1$.
\end{teo}

\begin{teo}
\label{T: Survival}
Let $\beta = \frac{1-\alpha}{\Delta-1}$ and assume $m \geq 2$ and $d \geq 2$ are fixed integers. We define the functions
% \vermelho{[Novas fórmulas -- favor conferir. 
% Talvez dizer no começo do Teorema (ou da Seção) que m e d são fixados e retirar m, d da notação $\fA$, $\fgamma$, etc]:}
{\allowdisplaybreaks
\begin{equation}
\label{F: function B}
\begin{aligned}
\fphi&= -d(m-1)r^2 + 2d(m-1)r,\\[4pt] 
\fA &= 
\frac{
1 - \beta p (m-2)-\sqrt{[1 - \beta p (m-2)]^2 - 4 \alpha \beta p^2 d (m-1)
}}{2 \beta p d (m-1)},\\[10pt]
\fgamma &= 1 - \beta p \, [d (m-1) \fA + m - 3],\\[10pt]
\fQ &= -\alpha p \, [\fgamma + \beta p (m-2)],\\[10pt]
\fR &= \fgamma \, [1 - \beta p (d-1) (m-1) \fA] - \beta p^2 (\alpha + \beta) (m-2),\\[10pt]
\fS &= -\beta p \, [\fgamma + \beta p (m-2)],\\[10pt]
\fB &= \frac{-\fR + \sqrt{\fR^2-4 \, \fQ \, \fS}}{2 \, \fQ}.
\end{aligned}
\end{equation}}%
% \vermelho{[conferir as soluções de A, B e C]:}
% {\allowdisplaybreaks
% \begin{equation}
% \label{F: functions ABC}
% \begin{aligned}
% \fphi&= -d(m-1)r^2 + 2d(m-1)r\\ %\vermelho{\text{survival function (directly by BP or via renewal equation)}}\\
% \fA &= 
% \frac{
% 1 - p(m-2)\beta-\sqrt{\big[1 - p(m-2)\beta\big]^2 - 4p^2 d(m-1)\beta\alpha
% }}{2p d(m-1)\beta},\\[10pt]
% \fC &= 
% \frac{(\beta + \alpha \fB)p}
% {1 - p\beta\big[(m-3) + d(m-1)\fA\big]},\\[10pt]
% \fxi &= 1 - p\beta[(d-1)(m-1)\fA + (m-2)\fC],\\[10pt]
% \fB &= \frac{\fxi-\sqrt{\fxi^2 - 4p^2\alpha\beta[1+(m-2)\fC]}}{2p\alpha},
% \end{aligned}
% \end{equation}}%
% Also consider the set
% \[ \cSurv = \{(\alpha, p):  \phi(\fB)> 1\}. \]
Then, \(\FM{\Dmd}{\alpha,p}\) survives with positive probability for  $(\alpha, p)$ such that $\phi(\fB)> 1$.
\end{teo}

% To sum up, phase diagrams for $(m,d) \in \{(2, 2), (2, 20), (10, 2), (10, 20)\}$ are illustrated in Figure~\ref{Fig: PDiag-all}, where we represent respectively in pink and in blue the region of almost sure extinction and the region of survival with positive probability. The region in white is not covered by our results.

To summarize, the phase diagrams for $(m,d) \in {(2,2), (2,20), (10,2), (10,20)}$ are presented in Figure~\ref{Fig: PDiag-all}. The regions corresponding to almost sure extinction and survival with positive probability are depicted in pink and blue, respectively, while the white region represents parameter values for which our results do not provide a conclusion.

\begin{figure}[htbp]
\centering
\begin{subfigure}[b]{0.45\textwidth}
\centering
\includegraphics[width=\linewidth]{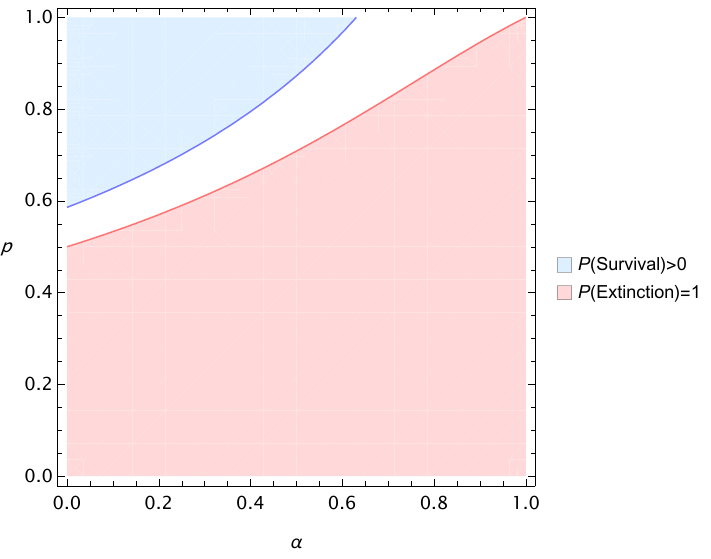}
\caption{$m = 2$ and $d = 2$.}
\label{Fig: PDiag-m2-d2}
\end{subfigure}
\hfill
\begin{subfigure}[b]{0.45\textwidth}
\centering
\includegraphics[width=\linewidth]{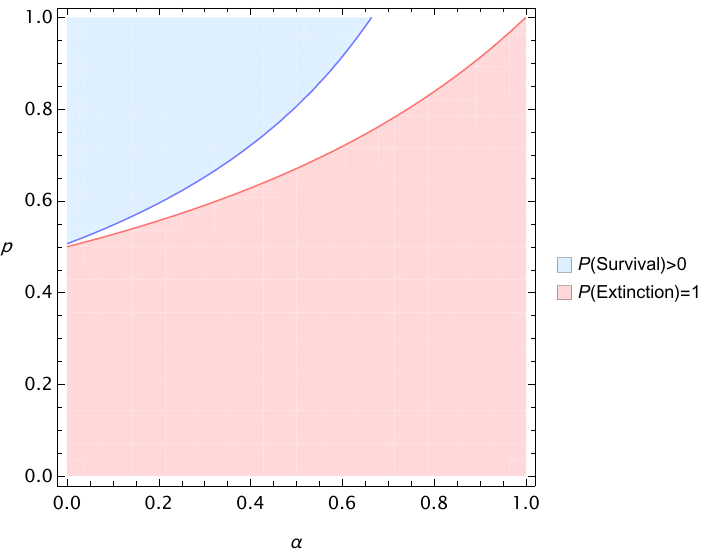}
\caption{$m = 2$ and $d = 20$.}
\label{Fig: PDiag-m2-d20}
\end{subfigure}

\vspace{0.5cm}

\begin{subfigure}[b]{0.45\textwidth}
\centering
\includegraphics[width=\linewidth]{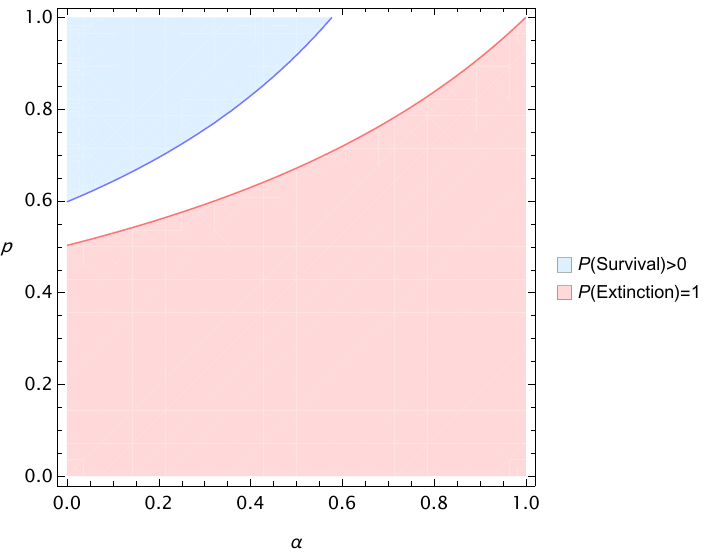}
\caption{$m = 10$ and $d = 2$.}
\label{Fig: PDiag-m10-d2}
\end{subfigure}
\hfill
\begin{subfigure}[b]{0.45\textwidth}
\centering
\includegraphics[width=\linewidth]{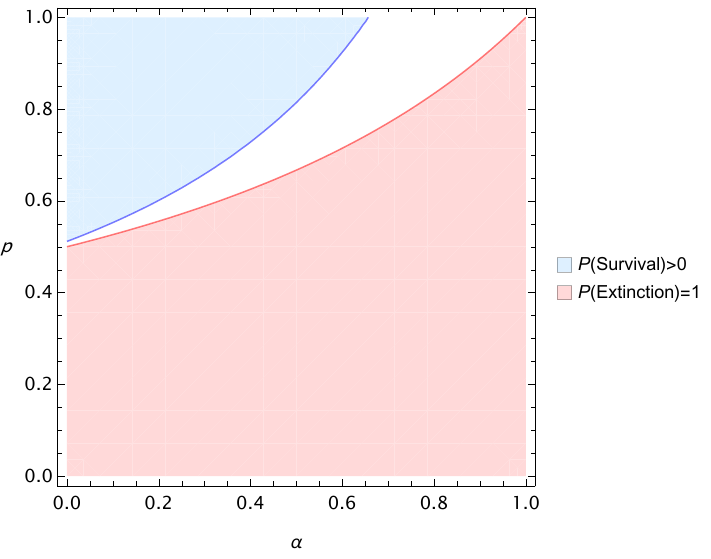}
\caption{$m = 10$ and $d = 20$.}
\label{Fig: PDiag-m10-d20}
\end{subfigure}
\caption{Phase diagrams for the $\FM{\Dmd}{\alpha,p}$.}
\label{Fig: PDiag-all}
\end{figure}

% Versão com 3:
% \begin{figure}[htbp]
%      \centering
%      \begin{subfigure}[b]{0.3\textwidth}
%             \centering
%     \includegraphics[width=1\linewidth]{Fig-Drift-PDiag-m10-d2.pdf}
%     \caption{$m = 10$ and $d = 2$.}
%     \label{Fig: PDiag-m10-d2}
%      \end{subfigure}
%      \hfill
%      \begin{subfigure}[b]{0.3\textwidth}
%          \centering
%     \includegraphics[width=1\linewidth]{Fig-Drift-PDiag-m10-d20.pdf}
%     \caption{$m = 10$ and $d = 20$.}
%     \label{Fig: PDiag-m10-d20}
%      \end{subfigure}
%      \hfill
%      \begin{subfigure}[b]{0.3\textwidth}
%          \centering
%     \includegraphics[width=1\linewidth]{Fig-Drift-PDiag-m2-d20.pdf}
%     \caption{$m = 2$ and $d = 20$.}
%     \label{Fig: PDiag-m2-d20}
%      \end{subfigure}
%      \caption{Phase diagrams for the $\FM{\Dmd}{\alpha,p}$.}
%      \label{Fig: PDiag-all}
% \end{figure}

% \begin{figure}
%     \centering
%     \includegraphics[width=0.5\linewidth]{Fig-Drift-PDiag-m10-d2.pdf}
%     \caption{Phase diagram for $m = 10$ and $d = 2$.}
%     \label{Fig: PDiag-m10-d2}
% \end{figure}

% \begin{figure}
%     \centering
%     \includegraphics[width=0.5\linewidth]{Fig-Drift-PDiag-m10-d20.pdf}
%     \caption{Phase diagram for $m = 10$ and $d = 20$.}
%     \label{Fig: PDiag-m10-d20}
% \end{figure}

% \begin{figure}
%     \centering
%     \includegraphics[width=0.5\linewidth]{Fig-Drift-PDiag-m2-d20.pdf}
%     \caption{Phase diagram for $m = 2$ and $d = 20$.}
%     \label{Fig: PDiag-m2-d20}
% \end{figure}

\subsection*{The symmetric case: $\alpha = 1/\Delta$}

First, note that by a direct application of Proposition 1.2 in \citet{AMP-PT}, we have
\begin{align}\label{LB: AMP(2002)}
    p_c(\Dmd) \geq \frac{(d+1)(m-1)}{2[(d(m-1)+(m-2)] + 1}.
\end{align}

However, by comparison with a Galton--Watson branching process (see Section \ref{SS: BPT}) we can improve the bound in \eqref{LB: AMP(2002)} for the frog model on $\Dmd$, obtaining the next result.

\begin{teo}\label{T: LBBPT} In the symmetric case, 
    for any $m \geq 2$ and $d \geq 2$, 
    \begin{align} \label{eq:3.4}
        p_c(\Dmd) \geq \frac{\sqrt{[2d(m-1)+(m-2)]^2 + 4d(m-1)} - (m-2)}{4d(m-1)}.
    \end{align}
\end{teo}

On the other hand, the following result gives rigorously determined upper and lower bounds for the critical probability of $\FM{\Dmd}{p}$.

% \begin{defn}\label{Def: Bounds}
% For any $m \geq 2$ and $d \geq 2$, we define

% \[\underaccent{\bar}{p}(m,d) =\frac{(d+1)(m-1)\underaccent{\bar}{r}}{d(m-1)\underaccent{\bar}{r}^2 + (m-2)\underaccent{\bar}{r} + 1},
% \]
% where $\underaccent{\bar}{r}$ is the unique root in $[0,1/d(m-1)]$ of the polynomial
% \begin{align}
% d(m-1)r^2 + [2d(m-1) + (m-2)]r - 1 = 0, 
% \end{align}
% and
% \[
% \bar{p}(m,d) = \frac{(d+1)(m-1)\bar{r}}{d(m-1)\bar{r}^2 + (m-2)\bar{r} + 1},
% \]
% where $\bar{r}$ is the unique root in $[0,1/d(m-1)]$ of the polynomial 
% \[
% d(m - 1)r^2 - 2d(m - 1)r + 1 = 0.
% \]
% \end{defn}

% \begin{align}
% \underaccent{\bar}{p}(m,d) &= \frac{\sqrt{[2d(m-1)+(m-2)]^2 + 4d(m-1)} - (m-2)}{4d(m-1)}\\
% \bar{p}(m,d) &= \frac{(d+1)(m-1)}{2d(m-1) + (m-2)}
% \end{align}

% \vermelho{limitante inferior obtido com o TTBP mais básico}
% \begin{teo} \label{T: LB-TTBP-Basico}
%     For any $m \geq 2$ and $d \geq 2$,
% \[ \underaccent{\bar}{p}(m,d) \leq p_c(\Dmd) \leq \bar{p}(m,d), \]
% where
% \[ \underaccent{\bar}{p}(m,d) = \left\{\begin{array}{cl}
% \frac{2(d+1)}{\sqrt{4d^2+4d-3} + 2d + 1} &\text{if } m = 2,\\[0.2cm]%
% (3.2) &\text{if } m \geq 3,
% \end{array} \right. \]
% \end{teo}

%\vermelho{limitante inferior obtido com o TTBP mais geral (melhora o anterior)}
\begin{teo}
\label{T: bounds}
In the symmetric case, for any $m \geq 2$ and $d \geq 2$,
\[ \underaccent{\bar}{p}(m,d) \leq p_c(\Dmd) \leq \bar{p}(m,d), \]
where
\[ \underaccent{\bar}{p}(m,d) = \frac{2(d+1)(m-1)}{2d(m-1)+2m-3+\sqrt{\left(2d(m-1)+2m-3\right)^2-4(2m-3)}}\]
%\left\{\begin{array}{cl}
%\frac{2(d+1)}{\sqrt{4d^2+4d-3}+2d+1} &\text{if } m = 2,\\[0.3cm]%
%\frac{\sqrt{[2d(m-1)+(m-2)]^2 + 4d(m-1)} - (m-2)}{4d(m-1)} &\text{if } m \geq 3,
%\end{array} \right. 
%\]
and
\[ \bar{p}(m,d) = \frac{(d+1)(m-1)}{2d(m-1) + (m-2)}. \]
\end{teo}

\begin{obss} Note that
\begin{itemize}
    \item[(i)] When $m=2$,  we recover the lower and the upper bounds given in \citet{GP} (see Theorem 1 therein) and \citet{LMP-IUB} (see Theorem 4.1 therein), respectively. 
    \item[(ii)] $\underaccent{\bar}{p}(m,d)$ is better than the lower bound in \eqref{eq:3.4}.
    \item[(iii)] Theorem \ref{T: bounds} follows as a corollary of Theorems \ref{T: Extinction} and \ref{T: Survival} by setting $\alpha = 1/\Delta$.
\end{itemize}
\end{obss}

To establish our main results, we couple the frog model with suitably defined branching processes, specifically, a supercritical to prove survival and subcritical one to establish extinction.

For clarity of exposition, we split the analysis into two parts: Section~\ref{S: Proofs Symmetric} addresses the symmetric case ($\alpha = 1/\Delta$), while the general case is treated in Section~\ref{S: Proofs General}.

%For any $m \geq 2$ and $d \geq 2$,
%\[
% \frac{\sqrt{[2d(m-1)+(m-2)]^2 + 4d(m-1)} - (m-2)}{4d(m-1)} \leq p_c(\Dmd) \leq \frac{(d+1)(m-1)}{2d(m-1) + (m-2)}.
%\underaccent{\bar}{p}(m,d) \leq p_c(\Dmd) \leq \bar{p}(m,d), 
%\] 
%where $\underaccent{\bar}{p}(m,d)$ and $\bar{p}(m,d)$ are given in Definition \ref{Def: Bounds}.

%\begin{teo}
%\label{T: UB}
%For any $m \geq 2$ and $d \geq 2$,
%\[
%p_c(\Dmd) \leq \bar{p}(m,d).
%\]
%\end{teo}

% \begin{cor} For any $d \geq 2$,
% \[ 
% GP(2023)  \leq p_c(\Td) \leq \frac{d+1}{2d}, \]
% %\begin{align}
% % \frac{1}{2}\sqrt{1+ \frac{1}{d}} \leq p_c(\Td) \leq \frac{d+1}{2d}.
% %\end{align}
% \end{cor}

% \begin{proof}
% Taking $m=2$ in Theorem \ref{T: bounds} and after some elementary calculations, the result follows. 
% \end{proof}
% \vermelho{Asymptotics for $p_c$:}
% \begin{cor} 
% \[
% \frac{1}{2}\leq \lim_{m \to \infty} p_c(\Dmd) \leq \frac{d+1}{2d}
% \]

% \[
% \lim_{d \to \infty}\lim_{m \to \infty} p_c(\Dmd) = \frac{1}{2}
% \]
% \end{cor}

\section{Proofs: Symmetric random walks}
\label{S: Proofs Symmetric}

Although Theorems~\ref{T: LBBPT} and \ref{T: bounds} follow directly from Theorems~\ref{T: Extinction} and \ref{T: Survival} by setting $\alpha = 1/\Delta$, we begin by considering the frog model on $\Dmd$ with $\alpha = 1/\Delta$ for the sake of clarity and to introduce the main ideas.  We first present several preliminary results that will be needed later. In particular, we require some facts concerning the hitting probabilities of random walks on $\Dmd$.
% \section{Proofs}
% In this section, we prove the main results of the paper. To show survival and extinction, we compare the frog model with supercritical and subcritical branching processes, respectively.

\subsection{Preliminaries}

% To ease the understanding, we split the analysis into two parts: I. $\alpha = 1/\Delta$, and II. $(\alpha,p) \in [0,1]^2$.

% \subsection*{I. Symmetric random walks}

The following result provides a formula for the probability that a frog starting at a vertex $u$ visits $v$ in the event that it is activated (equivalently, the probability that the oriented edge from $u$ to $v$ is open).

\begin{lem}
\label{L: Prob Open}
Let $u,v \in \Dmd$ with \(\dist(u, v) = n \geq 1\).
Then,
\[ \bbP(\s{u}{v}) = [\fr]^n, \]
where 
\begin{equation}
\label{E: Prob Open Symmetric}
\fr = \frac{\Delta-p(m-2) - \sqrt{[\Delta-(m-2)p]^2 - 4(m-1)dp^2}}{2(m-1)dp}.
\end{equation}

\end{lem}

\begin{proof}
Let \(u\) and \(v\) vertices in $\Dmd$ such that $\dist(u,v) = n$ and \(\tau_{uv}\) be the first time when a SSRW on \(\Dmd\) starting from \(u\) visits \(v\). Conditioning on the lifetime of the frog at $u$, we have for \(p < 1\), 
\begin{equation*}
\bbP(\s{u}{v}) = \sum_{\ell = n}^{\infty} \bbP(\s{u}{v} \dado \tau_{uv} = \ell) \, \bbP(\tau_{uv} = \ell) 
= \bbE(p^{\tau_{uv}}).
\end{equation*}

Now, we denote the unique geodesic from $u$ to $v$ by $u=u_0, u_1, \ldots, u_n =v$. Since each $u_j$ ($j=0,\ldots,n$) is a cut-vertex we note that $\tau_{uv}$ is the sum of $n$ i.i.d.random variables, each with the distribution of $\tau_{u_0u_1}$. Hence, $\bbP(\s{u}{v}) = [\bbE(p^{\tau_{uu'}})]^n$, where $u\sim u'$. But by conditioning on the first step of the random walk, we obtain
\begin{align}
\label{F: Eq-FGP}
\bbE(p^{\tau_{uu'}}) = \frac{p}{\Delta} + p\left(\frac{m-2}{\Delta}\right)\bbE(p^{\tau_{uu'}}) + p\left( \frac{d(m-1)}{\Delta}\right)[\bbE(p^{\tau_{uu'}})]^2.
\end{align}
Solving equation \eqref{F: Eq-FGP} and using the fact that the probability generating function of $\tau_{uu'}$ satisfies $\lim_{p \to 0^+} \bbE(p^{\tau_{uu'}})= 0$, we get
\[
\bbE(p^{\tau_{uu'}}) = r(p,m,d).
\]
This finishes the proof for $p<1$. To conclude, for $p=1$, we have
\[ \bbP(\s{u}{v}) = \bbP(\tau_{uv}<\infty) = \lim_{p\to 1^-} \bbE(p^{\tau_{uv}}) =\left[\frac{1}{d(m-1)}\right]^n. \qedhere\]
\end{proof}

In the sequel, we will make use of the following result that establishes an equivalence between the function $\fr$ given in \eqref{E: Prob Open Symmetric} and $p$. 
\begin{lem}
\label{L: Equiv}
For each fixed $m \geq 2$ and $d \geq 2$, the function $r(p,d,m)$ is increasing in $p$. Moreover, for $\nu \in \left[0,\frac{1}{d(m-1)}\right]$, 
\[
r(p,d,m) = \nu \iff p = \frac{(d+1)(m-1)\nu}{d(m-1)\nu^2 + (m-2)\nu + 1}.
\] 
\end{lem}

\subsection*{Extinction} \label{SS: Extinction} In the following sections, we prove Theorem~\ref{T: LBBPT} and establish the lower bound in Theorem~\ref{T: bounds} by dominating the frog model with suitable subcritical branching processes. This approach yields sufficient conditions for the frog model on $\Dmd$ to die out almost surely when $p$ is sufficiently small.

\subsection{Frog model dominated by a simple branching process}\label{SS: BPT}
Define $X_{\raiz} = |\cR_{\raiz}|$, and for a vertex $v \neq \raiz$, let $v^-$ be the unique neighbor vertex of $v$ on the geodesic to the root. Consider the random variable
\begin{align*}
X_v = |\cR_v|- \ind\{\s{v}{v^-}\}.
\end{align*}

First, it is straightforward to show that for all $p < 1$ we have $\bbE(X_{\raiz}) < \infty$ and therefore $X_\raiz$ is finite almost surely.   
Now, consider a Galton--Watson branching process whose initial generation is distributed as $X_{\raiz} = |\cR_{\raiz}|$ and its family size is defined for $X_v$. Observing that an active frog placed at $v$ does not hit an inactive particle in its first step if it moves toward its parent vertex $v^-$, one gets that the frog model on $\Dmd$ (seen as a percolation model) is dominated by the Galton--Watson process just defined.

\begin{proof}[Proof of Theorem \ref{T: LBBPT}]
Let $m \geq 2$ and $d \geq 2$. 
By definition of the branching process described in Section~\ref{SS: BPT}, we have that the $\FM{\Dmd}{p}$ dies out almost surely if 
\begin{equation*}
\begin{aligned}
\bbE(X_v) &= \bbE(X_\raiz) - \bbP(\s{v}{v^-})\\
&=\left(\frac{d+1}{d}\right)\sum_{n\geq 1} [d(m-1)r]^n - r\\
&=\frac{r[d(m-1)(r+1) + (m-2)]}{1-d(m-1)r} < 1,
\end{aligned}
\end{equation*}
where $r = \fr$ is as in Lemma \ref{L: Prob Open}. It follows from Lemma \ref{L: Equiv}  and simple computations that if
\begin{equation*}
    p < \frac{\sqrt{[2d(m-1)+(m-2)]^2 + 4d(m-1)} - (m-2)}{4d(m-1)},
\end{equation*}
then the Galton--Watson process defined above is subcritical, therefore the $\FM{\Dmd}{p}$ dies out almost surely. 
\end{proof}

\subsection{Modified frog model}
\label{SS: FMDTTBP} 
Here are the main ideas for the proofs of the lower bound stated in Theorem~\ref{T: bounds} and also for Theorem~\ref{T: Extinction}, as we will see later. 
Instead of letting all active frogs move simultaneously, we consider a one-frog-at-a-time dynamics, in which at each time step only one active frog is selected and either dies or makes one move.
This new version of the frog model has the same extinction probability as the original model. 
After a random time, the set \(\cvis_t\) of visited vertices at time~\(t\) contains exactly $\Delta + 2$ vertices.
From that moment on, we classify active frogs, depending of its location, as either Type~\(1\) (if it is located at a vertex in which only his parent vertex was already visited) or Type~\(2\) (if it is located at a vertex with at least two already visited sites), we denote this model by $\MFM{\Dmd}{p}$ (recall that we assume $\alpha = 1/\Delta$ and therefore we omit the dependence on $\alpha$).
Then, we construct a two-type branching process that dominates the frog population at all times. The proof is completed by showing that the two-type branching process is subcritical. For the frog model in \(\Td\) with death, but no drift (i.e.\ \(\alpha = 1 / (d + 1)\)), \citet{GP} use this approach to derive a lower bound for the critical parameter \(p_c(\Td)\) of survival.

% In what follows, we keeping the notation of \citet{GP}, whenever possible. Let $K = \inf\{t \geq 0: |\cvis_t| = \Delta + 2\}$ and fix $t \geq K$ and $\cvis_t=T$ and denote $p_v(j,k|T)$ the probability that the frog located at site $v$ inside $T$ (which has been chosen to acts)  produces $j$ particles of type 1, and $k$ particles of type 2. For every $v$, $p_v$ takes values on $\Omega = \{(0,0), (1,0),(0,1),(2,0)\}$.
% The location of $v$ inside $T$ determines the type of the frog:
% \begin{itemize}
%     \item [(i)] If the frog has type 1. Then, independently of $v$ and $T$
%     \begin{equation}
%         p_v(j,k|T) = p_1(j,k) \quad \text{for all} \quad (k,j) \in \Omega;
%     \end{equation}
%     \item[(ii)] If the frog has type 2. Then, there exists integers $a,b$, with $a \geq 1$ and $a+b \geq 2$, such that
%     \begin{equation}
%     \begin{aligned}
%          p_v(0,0|T) &= 1-p, \\
%          p_v(0,1|T) &= \frac{pa}{\Delta},\\
%          p_v(1,0|T) &= \frac{pb}{\Delta}, \\
%          p_v(2,0|T) &= p\left[\frac{\Delta - (a+b)}{\Delta}\right].
%          \end{aligned}
%     \end{equation}
% \end{itemize}

\subsection{$\MFM{\Dmd}{p}$ dominated by a two-type branching process} 

Let us define a two-type branching process, which we refer to as the Symmetric Two-Type Branching Process (STTBP) to emphasize that we are working within the symmetric framework. For $i=1,2$, we denote by $p_i(j,k)$ the probability that a type $i$ particle produces $j$ particles of type 1, and $k$ particles of type 2. 

Formally, for $m \geq 2$ and $d \geq 2$, let $\Delta = (d+1)(m-1)$ and define
\begin{equation*}
   \begin{aligned}
p_1(0,0) &= 1-p,\\
p_1(1,0) &= 0,\\ 
p_1(2,0) &= p\left(\frac{d(m-1)}{\Delta}\right),\\ 
p_1(0,1) &= \frac{p}{\Delta},\\
p_1(0,2) &= p\left(\frac{m-2}{\Delta}\right),\\
\end{aligned} \quad
\begin{aligned}
    p_2(0,0) &= 1-p,\\
    p_2(1,0) &= \frac{p}{\Delta},\\
    p_2(2,0) &= p\left(\frac{d(m-1)-1}{\Delta}\right),\\
    p_2(0,1) &= \frac{p}{\Delta},\\ 
    p_2(0,2) &= p\left(\frac{m-2}{\Delta}\right).
\end{aligned}
\end{equation*}

Now, let $M$ be the first moment matrix of the STTBP. Since the number of types is finite, it is well known (see \citealp[Chapter~V]{AN-BP}) that the multi-type Galton--Watson process dies out almost surely if, and only if, the largest eigenvalue of $M$ is less than or equal to 1. In our case, the matrix $M$ is given by
\[
M = \begin{bmatrix}
    \frac{2d(m-1)p}{\Delta} & \frac{p(2m-3)}{\Delta}\\
    \frac{(2d(m-1)-1)p}{\Delta} & \frac{p(2m-3)}{\Delta}
\end{bmatrix},
\]
and its largest eigenvalue is 
\[
\psi(M) = \frac{\left(2d(m-1)+2m-3+\sqrt{\left(2d(m-1)+2m-3\right)^2-4(2m-3)}\right)p}{2(d+1)(m-1)}.
\]
Consequently, if
\begin{equation*}
   p < \frac{2(d+1)(m-1)}{2d(m-1)+2m-3+\sqrt{\left(2d(m-1)+2m-3\right)^2-4(2m-3)}},
   %\frac{2 (d+1) (m-1)}{2 (d+1) (m-1)+\sqrt{4 (d+1)^2 (m-1)^2-4 (d+1)(m-1)-3}-1},
\end{equation*}
then the STTBP becomes extinct almost surely.
% $\bbP(\FM{\Dmd}{p} \text{ survives}) = 0$. 
Hence, the lower bound in Theorem~\ref{T: bounds} is established once we prove that the STTBP dominates the modified frog model.
%This completes the proof of the lower bound in Theorem \ref{T: bounds}.
% \begin{equation}\label{E: LB2TBP}
%     p_c(\Dmd) \geq \frac{2(d+1)(m-1)}{2d(m-1)+2m-3+\sqrt{\left(2d(m-1)+2m-3\right)^2-4(2m-3)}}
%     %\frac{2 (d+1) (m-1)}{2 (d+1) (m-1)+\sqrt{4 (d+1)^2 (m-1)^2-4 (d+1)(m-1)-3}-1}.
% \end{equation}

\subsection*{Coupling $\FM{\Dmd}{p}$ with a two-type branching process} 

In what follows, we keeping the notation of \citet{GP}, whenever possible. Let $K = \inf\{t \geq 0: |\cvis_t| = \Delta + 2\}$ and fix $t \geq K$ and $\cvis_t=T$ and denote $p_v(j,k|T)$ the probability that the frog located at site $v$ inside $T$ (which has been chosen to acts)  produces $j$ particles of type 1, and $k$ particles of type 2. For every $v$, $p_v$ takes values on $\Omega = \{(0,0), (1,0),(0,1),(2,0), (2,0)\}$.
The location of $v$ inside $T$ determines the type of the frog:
% Let $t \geq K =\inf\{t \geq 0: |\cvis_t| = \Delta + 2\}$ fixed. For a frog located at $v$ inside $T$, we have the following cases
\begin{itemize}
    \item [(i)] If the frog has type 1. Then, independently of $v$ and $T$,
    \begin{equation*}
        p_v(j,k|T) = p_1(j,k) \quad \text{for all} \quad (k,j) \in \Omega;
    \end{equation*}
    \item[(ii)] If the frog has type 2. Then, there exist integers $a,b,c$, with $a,b \geq 1$ and $a+b+c \geq 2$, such that
    {\allowdisplaybreaks
    % \begin{equation*}
    \begin{alignat*}{2}
         p_v(0,0|T) &= 1-p, &
         p_v(0,1|T) &= \frac{pa}{\Delta}, \\
         p_v(1,0|T) &=  \frac{pb}{\Delta}, &
         p_v(0,2|T) &= \frac{pc}{\Delta}, \\
         p_v(2,0|T) &= p\left(\frac{\Delta-(a+b+c)}{\Delta}\right). \quad& &
         \end{alignat*}}%
    % \end{equation*}
\end{itemize}
To couple the laws $p_1,p_2$ and $p_v(\cdot|T)$, we use three partitions of $[0,1]$ in intervals. 
Let
{\allowdisplaybreaks
% \begin{equation*}
   \begin{align*}
       \mathcal{Q}_1 &= \{J_1^{0,1}, J_1^{0,2},J_{1}^{2,0}, J_1^{0,0}\},\\
       \mathcal{Q}_2 &= \{J_2^{0,1}, J_{2}^{1,0},J_2^{0,2},J_{2}^{2,0},J_2^{0,0}\},\\
       \mathcal{Q}_{3} &=\{J_3^{0,1},J_3^{0,2},J_3^{1,0},J_3^{2,0},J_3^{0,0}\},
   \end{align*}}%
% \end{equation*}
where the intervals are given by
{\allowdisplaybreaks
\begin{alignat*}{2}
    J_1^{0,1} &= \left[0,\frac{p}{\Delta}\right), &
    J_1^{0,2} &= \left[\frac{p}{\Delta},\frac{p(m-1)}{\Delta}\right),\\
    J_1^{2,0} &= \left[\frac{p(m-1)}{\Delta},p\right), \qquad&
    J_1^{0,0} &= [p,1],\\[0.2cm]
    J_2^{0,1} &= \left[0,\frac{p}{\Delta}\right), &
    J_2^{1,0} &= \left[\frac{p}{\Delta},\frac{2p}{\Delta}\right),\\ J_2^{0,2} &= \left[\frac{2p}{\Delta},\frac{pm}{\Delta}\right), & J_2^{2,0} &= \left[\frac{pm}{\Delta},p\right),\\
    J_2^{0,0} &= [p,1], & &\\[0.2cm]
% \end{alignat*}
% \begin{alignat*}{2}
    J_3^{0,1} &= \left[0,\frac{ap}{\Delta}\right), &
    J_3^{0,2} &= \left[\frac{ap}{\Delta},\frac{(a+c)p}{\Delta}\right),\\ 
    J_3^{1,0} &=\left[\frac{(a+c)p}{\Delta},\frac{(a+b+c)p}{\Delta}\right), \:\:& 
    J_3^{2,0} &= \left[\frac{(a+b+c)p}{\Delta},p \right),\\ 
    J_3^{0,0} &= [p,1]. & &
\end{alignat*}}%
Furthermore, consider an independent sequence of i.i.d.\ random variables $\{U_t\}_{t \geq 0}$ with uniform distribution on [0,1]. 

We note that for all $(j,k) \in \Omega =\{(0,0),(0,1),(1,0),(0,2),(2,0)\}$ we have 
\begin{equation*}
    \bbP(U_{t+1} \in J_i^{j,k}) = p_i(j,k), \quad i =1,2,
\end{equation*}
and
\begin{equation*}
    \bbP(U_{t+1} \in J_3^{j,k}) = p_v(j,k|T).
\end{equation*}
Now, we couple both processes in such a manner that the two-type branching process has more particles than the modified frog model. 
We do this by updating both the processes at time $t+1$ using the same realization of the random variable $U_{t+1}$. Indeed, let $N_t^i$ and $M_t^i$ be the number of particles of type $i$ at time $t$ in the $\MFM{\Dmd}{p}$ and the STTBP, respectively. The $\MFM{\Dmd}{p}$ is started from any pair $(M_0^1,M_0^2)$ such that $|\cvis_0| = \Delta + 2$, and we set $(N_0^1,N_0^2) = (M_0^1,M_0^2)$. We show that, for all $t \geq 0$, the following inequalities hold:
    \begin{equation} \label{E: ineq STTBP}
         N_t^1 \geq M_t^1 \quad \text{and} \quad N_t^1 + N_t^2 \geq M_t^1 + M_t^2.
    \end{equation}
    The proof proceeds by induction on $t$. The base case $t=0$ follows immediately from our assumption on the initial configurations. For the inductive step, suppose that the inequalities in \eqref{E: ineq STTBP} hold for some $t \geq 0$; we then show that they remain valid for $t + 1$.
\begin{enumerate}
    \item [(i)] If the chosen frog in the $\MFM{\Dmd}{p}$ is of type 1, then a type 1 particle in the 2-type branching process is chosen to give birth to its offspring. In this case, the two chosen particles give birth to the same offspring.   
    \item [(ii)] If the chosen frog is of type 2, we have
    \begin{itemize}
        \item [(a)] The particle of the 2-type branching process is type 2 too, then
        \begin{itemize}
            \item if $U_{t+1} \leq \frac{p}{\Delta}$, then one particle of type 2 is created in both processes;
            \item if $U_{t+1} \geq p$, then both chosen particles dies in each process;
            \item if $U_{t+1} \in \left[\frac{p}{\Delta},\frac{ap}{\Delta}\right]$ (with $a \geq 2$) the 2-type branching process creates two type 2 particle, while the $\MFM{\Dmd}{p}$ creates one type 2 particle;   
            \item  if $U_{t+1} \in J_3^{0,2}$ then two particles type 2 are created in each process; 
            \item if $U_{t+1} \in J_3^{1,0}$, then in the 2-type branching process are created one type 1 (if $a = 1$) or two type 1 (if $a \geq 2$, $a+c \geq m$) or two type 2 (if $a \geq 2$, $a+c < m$), while in the $\MFM{\Dmd}{p}$ is created one particle type 1.
            \item if $U_{t+1} \in J_3^{2,0}$, then two particles type 1 are created in each process.
        \end{itemize}
        \item[(b)] The number of particles type 2 in the 2-type branching process is null, then we have to choose a type 1 particle. In this case, we follow the same reasoning as before, observing that $N_t^1 \geq M_t^1 +1$.
    \end{itemize}
\end{enumerate}

\subsection{Survival} \label{SS: Survival}
In this section, we prove the upper bound for $p_c(\Dmd)$ stated in Theorem~\ref{T: bounds}.
More precisely, we show that the frog model on $\Dmd$ survives with positive probability whenever the parameter $p$ is sufficiently close to $1$. Our approach is to construct a suitable sequence of branching processes and show that they are dominated by the frog model.

Roughly speaking, we consider the spreading of the frog model restricted to a tree rooted at $\raiz$ embedded in $\Dmd$, which is isomorphic to the $b$-ary tree with $b = d(m-1)$. To do this, we fix an arbitrary complete graph (\emph{clique} of order $m$) attached to the root, say $\Km^{(1)}$, and define $\Dmd^{+}(\raiz) = \Dmd \setminus \Km^{(1)}$. For a non-root vertex $u$, we denote by $\Dmd^{+}(u) = \{v \in \Dmd: u \leq v\}$ the set of \emph{descendants} of $u$. For $u \in \Dmd$ and $n \geq 1$, we denote by $B_n(u) = \{v \in \Dmd^{+}(u): \dist(u,v) = n\}$ the set of descendants of $u$ at distance $n$. Formally, denote by $\mathbb{T}_{b}^{+}\subset \Dmd$ the rooted $b$-ary tree defined inductively as follows: the root of $\mathbb{T}_b^+$ is $\raiz$ (the root of $\Dmd$), the children of the root are the vertices $u \in B_1(\raiz)$, for a vertex $u \in B_1(\raiz)$ the children of $u$ are the vertices in $B_1(u)$, and so on.
Now, in this restricted model, a frog originally at $v \in B_n(u)$ is activated from the frog located at $u$ when the event $\{\cam{u}{v}\}$ occurs (assuming the frog located at $u$ is active). For each $n \geq 1$, this process starting at the root can be seen as a Galton--Watson branching process, and whose survival implies the survival of the frog model on $\Dmd$. A similar approach was used in \citet{LU-PTBT} to establish an upper bound on the critical probability for the frog model with random initial configuration on biregular trees (see Definition 3.2 therein).
\begin{defn}
\label{D: UB-offspring}
For $u \in \Dmd$ and $v\in B_{n}(u)$, let $u_0=u < u_1 < \cdots < u_{n-1} < u_{n}=v$, be the unique geodesic connecting $u$ and $v$.
For each $\ell=1, 2, \dots, n-1$, we denote by $\{\m{u_0}{u_{\ell}}\}$ the event that $\{\s{u_0}{u_{\ell}}\} \cap \{\ns{u_0}{u_{\ell+1}}\}$.
We define $\{\cam{u_0}{u_{n}}\}$ inductively as follows
\[\{\cam{u_0}{u_{n}}\}:=\{\s{u_0}{u_{n}}\}\cup\bigcup_{\ell=1}^{n-1}\{\m{u_0}{u_{\ell}}, \cam{u_{\ell}}{u_{n}}\}, \text{for } n \geq 2,\]
with the initial condition $\{\cam{u_0}{u_1}\}:=\{\s{u_0}{u_1}\}$.
In addition, we denote the complement of $\{\cam{u_0}{u_n}\}$ by $\{\ncam{u_0}{u_n}\}$.
\end{defn}

With the notations introduced in Definition \ref{D: UB-offspring}, we define, for each $n \geq 1$, a branching process $(Z^{(n)}_\ell)_{\ell \geq 0}$ in which  $Z^{(n)}_0 = 1$, and whose family size distributed as 
\[
Z^{(n)} = \sum_{v \in B_n(u)} \ind\{\cam{u}{v}\}.
\]
\begin{lem} 
\label{L: Prob progeny}
We have  
\begin{align*}
\bbP(\cam{u_0}{u_n}) = U_n(r(p,m,d)),
\end{align*}
where $U_n(r)= r^n(2-r)^{n-1}$.
\end{lem}

\begin{proof}
First, note that the functions $U_n$ can be defined inductively by
\begin{align}\label{fnc U induct}
    U_n(r) = r^n + \sum_{\ell = 1}^{n-1}(r^\ell - r^{\ell + 1})U_{n-\ell}(r), \quad n \geq 1.
\end{align}

In fact, defining the functions $U_n$ as in \eqref{fnc U induct} we have $U_1(r) =r$ and $U_{n+1}(r)=r(2-r)U_n(r), \; n\geq 1$.

We proceed by induction on $n$. Clearly, the assertion is true for $n = 1$. Now, let $n \geq 2$ and $\ell = 1,\ldots,n-1$. Since $\{\s{u_0}{u_{\ell+1}}\} \subset \{\s{u_0}{u_{\ell}}\}$, we have  $\bbP(\m{u_0}{u_\ell}) = \bbP(\s{u_0}{u_\ell}) - \bbP(\s{u_0}{u_{\ell+1}})$. Thus, by Definition \ref{D: UB-offspring},
\begin{align*}
 \bbP(\cam{u_0}{u_n}) &= \bbP(\s{u_0}{u_n}) + \sum_{\ell=1}^{n-1} [\bbP(\s{u_0}{u_\ell})-\bbP(\s{u_0}{u_{\ell + 1}})]\bbP(\cam{u_\ell}{u_n}). 
\end{align*}
 Using Lemma \ref{L: Prob Open}, the result follows.
\end{proof}
% To show Lemma \ref{L: Prob progeny}, we proceed by induction on $n$; we omit the details here, since it is proved in \citet{LU-PTBT}. %For the full proof of Lemma \ref{L: Prob progeny} we refer the reader to \citet{LU-PTBT}.

\begin{obs}
Notice that, by finding $p$ such that each of $(Z^{(n)}_{\ell})_{\ell \geq 0}$ is supercritical, we get a sequence $(\bar{p}_n(m,d))_{n\geq 1}$ of upper bounds for $p_c(\Dmd)$. Thus, it will suffice to show that $\bar{p}_n(m,d)$ converges to the upper bound established in Theorem \ref{T: bounds}.
\end{obs}

Now we define, for $r \in \left[0,\frac{1}{d(m-1)}\right]$, the following functions
\begin{equation*}
\begin{aligned}
    u_n(r) &= [U_n(r)]^{1/n} - \frac{1}{d(m-1)},\\
    u_\infty(r) &= r(2-r) - \frac{1}{d(m-1)}.
\end{aligned}  
\end{equation*}

One can readily verify that $u_n$ and $u_\infty$ are increasing functions, with $u_n(0) < 0$ and $u_n(1/d(m-1))>0$. Thus, there exists a unique root $r_n \in (0, 1/d(m-1))$ for each $n$, as well as a unique root $r_\infty \in (0, 1/d(m-1))$ for $u_\infty$. Given the uniform convergence of $u_n$ to $u_\infty$, the sequence of roots $\{r_n\}$ converges to $r_\infty$ as $n \to \infty$. 

We are now in a position to conclude the proof of Theorem \ref{T: bounds}.

\begin{proof}[Proof of the upper bound in Theorem \ref{T: bounds}] 
It is clear that for each $n\geq 1$ the branching process $(Z^{(n)}_{\ell})_{\ell \geq 0}$ is below the frog model on $\Dmd$ and, by Lemma \ref{L: Prob progeny}, it has mean number of offspring per individual given by $\bbE(Z^{(n)}) = [d(m-1)]^nU_n(r(p,m,d))$. Hence, the $\FM{\Dmd}{p}$ survives with positive probability if
\begin{equation*}
 \bbE(Z^{(n)}) = [d(m-1)]^nU_n(r(p,m,d)) > 1.
\end{equation*}

So by solving the equation in $p$,  $u_n(r(p,m,d)) = 0$, we get a sequence of upper bounds for $p_c(\Dmd)$. Since $u_n(r) > 0$ for $r > r_n$, it follows from Lemma \ref{L: Equiv} that
\[
p > \frac{(d+1)(m-1)r_n}{d(m-1)r_n^2 + (m-2)r_n + 1} \Rightarrow \bbP(\FM{\Dmd}{p} \text{ survives}) > 0.
\]
Therefore,
\begin{equation} \label{E: seq UBs}
    p_c(\Dmd) \leq \frac{(d+1)(m-1)r_n}{d(m-1)r_n^2 + (m-2)r_n + 1} = \bar{p}_n(m,d).
\end{equation}
Taking the limit as $n \to \infty$ in the right-hand side of \eqref{E: seq UBs}, the proof is complete.
\end{proof}

\section{Proofs: General case}
\label{S: Proofs General}

% \subsection*{II. General case}

Now we deal with the general model, by assuming that active particles move according to $\alpha$-biased random walks as in \eqref{F: drift-RW}. 
Here, we are interested in studying for which values of $(\alpha,p) \in [0,1]^2$ the frog model survives with positive probability or dies out almost surely. 
To prove Theorems~\ref{T: Extinction} and \ref{T: Survival}, we follow similar steps to those used in Section~\ref{S: Proofs Symmetric}, so we explain the adaptations of the reasoning. 

For the extinction part, we consider again the modified frog model in which a single frog acts (either dying or moving) at each time step. Recall that this does not change the extinction probability of the model, and allows us to couple this modified version together with a dominating subcritical branching process with two types.

\subsection{$\MFM{\Dmd}{\alpha,p}$ dominated by a two-type branching process}\label{SS: alpha-TTBP}

Let us remember that $m \geq 2$, $d \geq 2$, $\Delta = (d+1)(m-1)$, and $\beta = \frac{1-\alpha}{\Delta -1}$. 
We consider a two-type branching process whose offspring distribution is given by
\begin{equation*}
   \begin{aligned}
p_1(0,0) &= 1-p,\\
p_1(1,0) &= 0,\\
p_1(2,0) &= pd(m-1)\beta,\\
p_1(0,1) &= p\alpha,\\
p_1(0,2) &= p(m-2)\beta,
\end{aligned} \qquad
\begin{aligned}
    p_2(0,0) &= 1-p,\\
    p_2(1,0) &= p\beta,\\
    p_2(2,0) &= p[d(m-1)-1]\beta,\\
    p_2(0,1) &= p\alpha,\\
    p_2(0,2) &= p(m-2)\beta.
\end{aligned}
\end{equation*}
We refer to this model as the $\alpha$-TTBP.

Now, the first moment matrix is
\[
M' = \begin{bmatrix}
    2d(m-1)\beta p & (\alpha + 2(m-2)\beta)p\\
    (2d(m-1)-1)\beta p & (\alpha + 2(m-2)\beta)p
\end{bmatrix},
\]
and it is straightforward to show that the largest eigenvalue of $M'$ is the function $\fpsi$ given in~\eqref{F: LE}.
% \[ \psi(M') = \frac{p}{2} \left(2-\alpha +\sqrt{(2-\alpha)^2-\frac{4 (1-\alpha)((m-2)(2-\alpha)+\alpha d(m-1))}{(d(m-1)+m-2)^2}}\right). \]
Thus, we are ready to prove Theorem \ref{T: Extinction}.

\begin{proof}[Proof of Theorem \ref{T: Extinction}]
Analysis similar to that in Section~\ref{S: Proofs Symmetric} shows that the frog model with drift is dominated by the $\alpha$-TTBP defined in Section \ref{SS: alpha-TTBP}.
So the frog model dies out almost surely for $(\alpha,p)$ such that $\fpsi < 1$.
\end{proof}

The following result is a generalization of Lemma \ref{L: Prob Open} for the frog model with death and drift.
 
\begin{lem}
\label{L: Prob Open Drift}
Let \(u\) and \(v\) be any non-root vertices of \(\Dmd\) with \(u < v\), \(\dist(u, \raiz) \geq 2\), and \(\dist(u, v) = n \geq 1\).
Then,
\[ \bbP(\s{u}{v}) = [\fB]^n, \]
where \(\fB\) is given in~\eqref{F: function B}.
\end{lem}

\begin{proof} 
Consider $u,v \in \Dmd$ with $\dist(u,v) = n$ and let $u=u_0 <\ldots < u_n=v$ be the vertices on the geodesic from $u$ to $v$. Let $\tau_{uv}$ be the first time when a $\alpha$-biased random walk  (moving according to \eqref{F: drift-RW}) starting from $u$ visits $v$. We have, for $p < 1$,
\[
\bbP(\s{u}{v}) = \bbE(p^{\tau_{uv}}).
\]
We note that $\tau_{uv}$ is a sum of $n$ independent copies of $\tau_{uu^+}$ in which $u^+$ is a fixed children neighbor of $u$. Thus, $\bbP(\s{u}{v}) = [\bbE(p^{\tau_{uu^+}})]^n$.
Now, we need the following probability generating functions 
\begin{equation*}
A = \bbE(p^{\tau_{uu^-}}), \quad B = \bbE(p^{\tau_{uu^+}}), 
\quad \text{and} \quad C = \bbE(p^{\tau_{u^*u}}), 
\end{equation*}
where $u^-, u^+$ and $u^*$ denote the parent vertex, a fixed child, and a fixed sibling of $u$, respectively. 

Let $\beta = \frac{1-\alpha}{\Delta - 1}$. By conditioning on the first step of the $\alpha$-biased random walk, we obtain the following non-linear system of equations:
\begin{equation}
\begin{aligned}
\label{F: Eq-FGP-Drift}
A &= p\alpha + p(m-2)\beta A + [pd(m-1)]\beta A^2,\\
    B &= p\beta + p(d-1)(m-1)\beta A B + p(m-2)\beta C B + p(m-2)\beta C + p\alpha B^2,\\
    C &=p\beta + p\alpha B + p(m-3)\beta C+pd(m-1)\beta A C.
\end{aligned}
\end{equation}

% \vermelho{Nova versão:}
Solving the first equation of~\eqref{F: Eq-FGP-Drift} and using the fact that the probability generating function of $\tau_{uu^-}$ satisfies $\lim_{p\to 0^-}\bbE(p^{\tau_{uu^-}}) = 0$, we get $A = \fA$ as defined in~\eqref{F: function B}.
From the third equation, it follows that
\begin{equation}
\label{F: function C}
C = \frac{p \, (\alpha B + \beta)}{\fgamma}.
\end{equation}
We observe that $\fgamma > 0$, so replacing \eqref{F: function C} in the second equation of~\eqref{F: Eq-FGP-Drift} and simplifying properly, we obtain
\begin{equation*}
\fQ \, B^2 + \fR \, B + \fS = 0.
\end{equation*}
Thus, using that $\lim_{p\to 0^-}\bbE(p^{\tau_{uu^+}}) = 0$, we conclude that $B = \fB$ is given by~\eqref{F: function B}.
% \[ 
% \bbE(p^{\tau_{uu^-}}) = \fA, \quad \bbE(p^{\tau_{uu^+}}) = \fB, \; \text{and} \; \bbE(p^{\tau_{uu^*}}) = \fC,
% \]
% where \(\fA\), \(\fB\) and \(\fC\) are defined in~\eqref{F: functions ABC}.
% The proof is finished for $p < 1$. 
% Finally, for $p = 1$, we have (\vermelho{done with the help of Gemini AI (conferir)}),
% \[\bbP(\s{u}{v}) =\bbP(\tau_{uv} < \infty) =\lim_{p \to 1^-}  \bbE(p^{\tau_{uv}})= 
% \left\{\begin{array}{cl}
%  \left(\frac{1}{d(m-1) - (m - 2)}\right)^n & \text{if } \alpha \leq \frac{1}{m},\\[12pt]  \left(\frac{(m-1)(1-\alpha)}{\alpha [d(m-1) + m - 2]}\right)^n  & \text{if } \alpha \geq \frac{1}{m}.
% \end{array}\right.
% \]
% \vermelho{Nova versão -- verificar. Se estiver correta, acho melhor a primeira fórmula.}
Finally, for $p = 1$, we compute with the software Mathematica that
\[\bbP(\s{u}{v}) =\bbP(\tau_{uv} < \infty) =\lim_{p \to 1^-}  \bbE(p^{\tau_{uv}})=
 \left(\min\left\{ \frac{\beta}{\alpha}, \frac{1}{d(m-1)} \right\}\right)^n.
\]
% \[\bbP(\s{u}{v}) =\bbP(\tau_{uv} < \infty) =\lim_{p \to 1^-}  \bbE(p^{\tau_{uv}})= 
% \left\{\begin{array}{cl}
%  \left(\frac{1-\alpha}{\alpha [d(m-1) + m - 2]}\right)^n & \text{if } \alpha \geq \frac{d (m-1)}{2 d (m-1)+m-2},\\[12pt]  
%   \left(\frac{1}{d (m-1)}\right)^n  & \text{otherwise}.
% \end{array}\right.
% \]
This completes the proof.
\end{proof}

The graphs of functions \(\fA\), \(\fB\) and \(\fC\) for $m=d=2$ are depicted in Figure~\ref{Fig: ABC}.
% related to the probability of an edge in the oriented frog model being open (seen as a percolation model); we refer the reader to Lemma~\ref{L: Prob Open Drift} for details.
% Their graphs for \(d = 2\) are depicted in Figure~\ref{Fig: ABC}

\begin{figure}[htbp]
     \centering
     \begin{subfigure}[b]{0.45\textwidth}
         \centering
         \includegraphics[width=\textwidth]{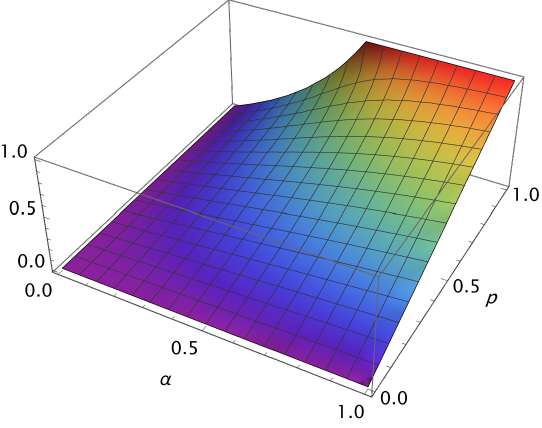}
         \caption{\(\fA\)}
         \label{fig:imagem1}
     \end{subfigure}
     \hfill
     \begin{subfigure}[b]{0.45\textwidth}
         \centering
         \includegraphics[width=\textwidth]{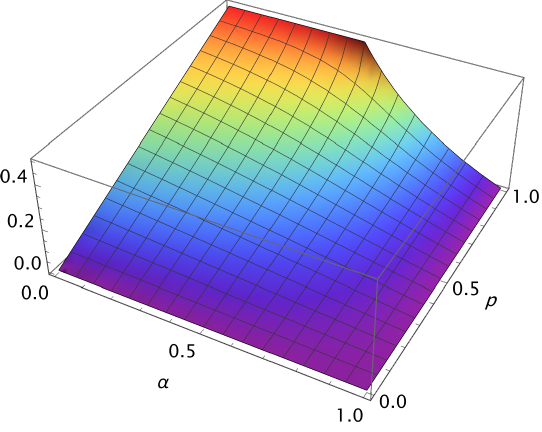}
         \caption{\(\fB\)}
         \label{fig:imagem2}
     \end{subfigure}
     
     \vspace{0.5cm}
     
     \begin{subfigure}[b]{0.45\textwidth}
         \centering
         \includegraphics[width=\textwidth]{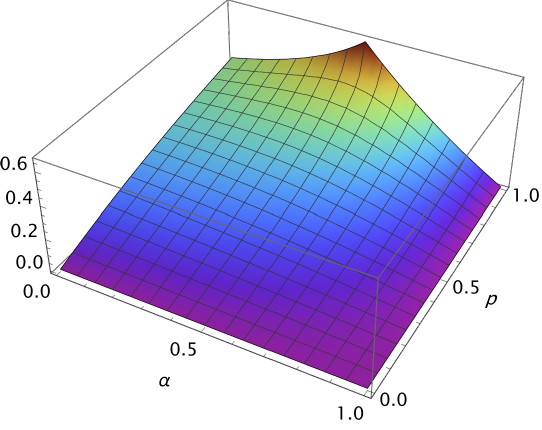}
         \caption{\(\fC\)}
         \label{fig:imagem3}
     \end{subfigure}
     \caption{Functions $\fA,\fB$ and $\fC$ for $m=d=2$.}
     \label{Fig: ABC}
\end{figure}

\vspace{.5cm}
We now turn to the proof of Theorem~\ref{T: Survival}.
 
\begin{proof}[Proof of Theorem~\ref{T: Survival}] 
  For each $n \geq 1$, let $(W_\ell^{(n)})_{\ell \geq 0}$ denote a branching process with $W_0^{(n)} = 1$, whose offspring distribution is given by 
  \[
  W^{(n)} = \sum_{v \in B_n(u)} \ind\{\cam{u}{v}\},
  \] 
where the event $\{\cam{u}{v}\}$ is given in Definition~\ref{D: UB-offspring}. In this case, it follows that the mean offspring number is given by
\[
\bbE(W^{(n)}) = [d(m-1)]^nU_n(\fB),
\]
where $U_n(r) = r^n(2-r)^{n-1}$ and $\fB$ is given in ~\eqref{F: function B}. By imposing that the Galton--Watson branching process is supercritical, the result follows.
\end{proof}

\bibliography{Bib-Frogs-Free-Product}
\bibliographystyle{plainnat}
\end{document}